\documentclass[leqno,10.5pt]{article} 
\usepackage{xcolor}
\usepackage{amsmath, amssymb}
\usepackage{amsthm} 
\usepackage{amscd}  
\usepackage{amsmath, amssymb} 
\usepackage{amsthm} 
\usepackage{ascmac} 
\usepackage{amscd}
\usepackage{color}

\newcommand{\ev}{{\rm Eval}}
\newcommand{\inte}{{\rm Prim}}
\newcommand{\deri}{{\rm Deri}}

\usepackage{ulem}
\theoremstyle{plain} 
\newtheorem{theorem}{\indent\sc Theorem}[section]
\newtheorem{lemma}[theorem]{\indent\sc Lemma}

\newtheorem{proposition}[theorem]{\indent\sc Proposition}

\theoremstyle{definition} 
\newtheorem{definition}[theorem]{\indent\sc Definition}
\newtheorem{facts}[theorem]{\indent\sc Facts}
\newtheorem{remark}[theorem]{\indent\sc Remark}
\newtheorem{example}[theorem]{\indent\sc Example}
\newtheorem{notation}[theorem]{\indent\sc Notation}

\newcommand{\C}{\mathbb{C}}
\newcommand{\R}{\mathbb{R}}

\newcommand{\Q}{\mathbb{Q}}

\newcommand{\Z}{\mathbb{Z}} 
\newcommand{\qu}{{\Q}}
\newcommand{\zu}{{\Z}}
\newcommand{\N}{\mathbb{N}}
\newcommand{\Li}{\rm{Li}}

\newcommand{\cd}{\cdots}

\def\cd{\cdots}

\def\2{I\hspace{-.1em}I}
\def\Li{\textrm{Li}}

\title{ Can  polylogarithms at algebraic points be linearly independent?}
\author{\textsc{Sinnou David}, \textsc{Noriko Hirata-Kohno} and \textsc{Makoto Kawashima}}
\date{}
\begin{document}

\maketitle

\hspace{5cm}{\it Dedicated to the memory of  Professor Naum Ilyitch Fel'dman}



\begin{abstract}
Let $r, \,m$ be positive integers. Let $0\leq x < 1$ be a rational number.
We denote by $\Phi_s(x,z)$ the $s$-th Lerch function {$\displaystyle\sum_{k=0}^{\infty}\frac{z^{k+1}}{{(k+x+1)}^s}$}
with $s=1, 2, \cd, r$. When $x=0$, this is the polylogarithmic function.
Let $\alpha_1, \cdots, \alpha_m$
be pairwise distinct algebraic numbers with $0<|\alpha_j|<1 \,\,\,(1\leq j \leq m)$.
We state a linear 
independence criterion over algebraic number fields of all the $rm+1$ numbers:~
$\Phi_1(x,\alpha_1)$, \,$\Phi_2(x,\alpha_1), \,$
$\cdots, \,  \Phi_r(x,\alpha_1)$, \,$\Phi_1(x,\alpha_2)$, \,$\Phi_2(x,\alpha_2), \,$
$\cdots , \, \Phi_r(x,\alpha_2), \cdots, \Phi_1(x,\alpha_m)$, 
~$\Phi_2(x,\alpha_m)$,\,$\cdots, \Phi_r(x,\alpha_m)$ and ~$1$. 
We obtain an explicit sufficient condition for
the linear independence of values of the $r$ Lerch functions $\Phi_1(x,z), \,\,\cd, \,\, \Phi_r(x,z)$ at $m$ {\it distinct points
in an algebraic number field of arbitrary finite degree} without any assumptions on $r$ and $m$.
When $x=0$, our result implies the linear independence of polylogarithms of distinct algebraic numbers of arbitrary degree,  subject to a metric condition.
We give an outline of our proof together with concrete examples of linearly independent polylogarithms.
\end{abstract}
\textit{Key words}:~
Lerch function, polylogarithms, linear independence, the irrationality, Pad\'e appro\-ximation.

\section{Introduction}
Let  $s$ be a  {non-negative} integer and $0\leq x < 1$ be a rational number.
We study the linear independence of values of the $s$-th Lerch function defined by
$$
\Phi_s(x,z) =\sum_{k=0}^{\infty}\frac{z^{k+1}}{{(k+x+1)}^s}~, \,\,\,\,\,\, z\in \C, \,\,\,  \vert z\vert <1\enspace.
$$
The $s$-th Lerch function $\Phi_s(x,z)$ satisfies the inhomogeneous differential equation:
\begin{align} \label{inhom diff eq}
\dfrac{d}{dz}\Phi_s(x,z)=\dfrac{1}{z}\Phi_{s-1}(x,z)-\dfrac{x}{z}\Phi_s(x,z){, \quad(s\geq 1)}.
\end{align}
Then the $s$-th Lerch function is a $G$-function in the sense of Siegel ({\it confer}  \cite{FN}, \cite{Siegel}).

Note that in the case of $x=0$, we have $\Phi_s(0,z)={\rm{Li}}_s(z)$ where 
\begin{equation*}
\Li_s(z) =\sum_{k=0}^{\infty}\frac{z^{k+1}}{{(k+1)}^s}~, \,\,\,\,\,\,z\in \C, \,\,\,  \vert z\vert <1,
\end{equation*}
is the $s$-th polylogarithmic function. 

\vspace{\baselineskip}
  
Let $r, \,m$ be positive integers and $K$ be an algebraic number field. 
Consider $\alpha_1,\ldots,\alpha_m\in K\setminus\{0\}$ with $\alpha_{i_1}\neq \alpha_{i_2}$ for $1\le i_1<i_2 \le m$ and $0 \le x\in \Q$.

\

We define the vector of formal power series $\vec{\Phi}$ by
$$\vec{\Phi}:={{}^{t}\kern-1pt(}1,\Phi_1(x,\alpha_1z),\ldots,\Phi_{r}(x,\alpha_1z),\ldots, \Phi_1(x,\alpha_mz),\ldots,\Phi_{r}(x,\alpha_mz))\in K[[z]]^{rm+1},$$
the vector of rational functions $\vec{A}(\alpha_i):={}^{t}(\alpha_i/(1-\alpha_i z),0, \ldots,0)\in K(z)^{r}$
and an $r\times r$ matrix $A(x)$ by
{$$A(x)=
\begin{pmatrix}
\tfrac{-x}{z} & 0 & \ldots & 0\\
\tfrac{1}{z}& \tfrac{-x}{z}  & \ldots & 0\\ 
\vdots    & \ddots &  \ddots & \vdots\\
0     &  \dots & \tfrac{1}{z} & \tfrac{-x}{z} 
\end{pmatrix}\quad (\text{if}\,\,\,r\geq 2), \quad A(x)=\left(\dfrac{-x}{z}\right) \quad (\text{if}\,\,\,r=1).$$}
Then, taking the differential equation~(\ref{inhom diff eq}) into account,  the vector $\vec{\Phi}$ satisfies the following system of differential equations in $\vec{y}$\,: 
\begin{align} \label{linear diff eq}
\dfrac{d}{dz}\vec{y}=\begin{pmatrix}
0 & 0 & \ldots & 0\\
\vec{A}(\alpha_1) & A(x) & \ldots & O\\
\vdots & \vdots & \ddots & \vdots\\
\vec{A}(\alpha_m) & O & \dots & A(x)
\end{pmatrix}\vec{y}\enspace.
\end{align} 
We see that (\ref{linear diff eq}) is indeed a system of {\it homogenous} differential equations in $\vec{y}$.

\vspace{\baselineskip}

We consider $r$ Lerch functions $\Phi_s(x,z), \,\,1\leq s \leq r$.
The linear independence of  $\Li_s(\alpha)$ at one rational number $\alpha $, with $1\leq s \leq r$, was studied by E.~M.~Niki\v{s}in \cite{N} in 1979.
It was generalized to the Lerch function by Kawashima \cite{Ka} and to algebraic cases by M. Hirose, M. Kawashima and N. Sato \cite{H-K-S}.
See also \cite{H-I-W} for examples.
In 1990, M. Hata \cite{Ha} adapted generalized Legendre polynomials modifying Pad\'e type
constructions of G.~V.~Chudnovsky  developed in \cite{ch2}, \cite{ch3}, \cite{ch9}, \cite{Chubrothers}, to obtain the linear independence of  $\Li_s(\alpha)$ 
  (indeed of the Lerch transcendent function)
for different $s$ but at one rational number $\alpha$. His result implies the irrationality of $\Li_2(1/q)$
with $q$ integer $q\geq 12$ whereas Chudnovsky announced in  \cite{ch2} the irrationality of $\Li_2(1/q)$  
with $q\geq 14$.
Later, Hata gave in 1993 the irrationality of the value of $\Li_2(1/q)$ in \cite{Ha1993}
with $q$ integer $q\geq 7$ or $q\leq -5$.

In 2005, Rhin and C. Viola \cite{RV1} adapted their permutation group method, established in 1996 \cite{RV0},
to get the irrationality of $\Li_2(\alpha)$ for certain $\alpha\in\Q$, involving 
the irrationality $\Li_2(1/q)$ with $q\geq 6$, $q\in\Z$ in qualitative andquantitative forms. More recently in 2018, 
Viola and W.~Zudilin \cite{VZ1} extended the permutation group method
with constructions 
to establish the linear independence of  $1, \, \Li_1(1/q)$, $\Li_2(1/q)$, $\Li_2(1/(1-q))$ over $\Q$
with an integer $q\geq 9$ or $q\leq -8$ and more generally,  that of
$1, \, \Li_1(\alpha)$, $\Li_2(\alpha)$, $\Li_2(\alpha/(\alpha-1))$ for certain $\alpha\in\Q$.
{See also important related works in \cite{FSZ} \cite{marc} \cite{Mi} \cite{Ri} \cite{Z}.}

With respect to logarithms, 
G.~Rhin and P.~Toffin \cite{R-T} created a system of
Pad\'e approximants to show the linear independence
of the natural logarithms of  {\it distinct}  $\alpha_1, \cdots, \alpha_m$, 
either rational or quadratic imaginary numbers, under a {\it metric} condition requiring the points $\alpha_1, \cdots, \alpha_m$
to be very close to
the origin $0$. This method provides a  refinement of previous lowed bounds for linear forms in logarithms,
especially for effective  bounds obtained by A. Baker \cite{Baker1975} and an essential improvement due to N. I. Fel'dman
\cite{fe1}, valid  under the above stated metric condition. 
This proof in \cite{R-T} opened a new path, albeit unexplored systematically, during the next decades  to show the linear independence of logarithms over $\Q$ at
{\it distinct} $\alpha\in\Q$, relying only on Pad\'e approximations.

Since $\Li_1(z)$ coincides with the usual natural logarithm, the Rhin-Toffin method suggests how to adapt 
Pad\'e approximations to deal with the linear independence of polylogarithms at distinct points
$\alpha_1, \cdots, \alpha_m$. 

\vspace{\baselineskip}

We give a new criterion
to show the linear 
independence of all the $rm+1$ numbers:~ $\Phi_1(x,\alpha_1)$, \,$\Phi_2(x,\alpha_1), \,$
$\cdots , \ , \Phi_r(x,\alpha_1)$, \,\,$\Phi_1(x,\alpha_2)$, \,$\Phi_2(x,\alpha_2), \,$
$\cdots , \, \Phi_r(x,\alpha_2), \cdots \cdots, \Phi_1(x,\alpha_m)$, 
~$\Phi_2(x,\alpha_m)$, ~$\cdots , \Phi_r(x,\alpha_m)$ and $1$, over an algebraic number field $K$, supposing
$\alpha_1, \cdots, \alpha_m$ pairwise distinct  in $K$, assumed to be
sufficiently close to the origin, which we will make precise later.
We also give an outline of our proof with basic ideas.

Our linear independence criterion for the values of the Lerch functions, including the case of
polylo\-ga\-rithmic functions, 
at {\it distinct points in an algebraic number field of arbitrary finite degree},
is not covered by the previous criterion in \cite{G1}, \cite{G2}, as is explained below in Remark $1.1.$, Remark $1.3.$, Remark $1.4.$  and Example $6.3$.

\begin{remark}\label{gal}
Let us describe here previous linear independence results concerning with values of the Lerch functions, at distinct rational or ima\-ginary quadratic points, due to
A.~I.~Galochkin \cite{G1}, \cite{G2}, Y. Z. Flicker, \cite{Fli} K.~V$\ddot{\text{a}}$$\ddot{\text{a}}$n$\ddot{\text{a}}$nen \cite{Va}, together with 
a result  by K.~V$\ddot{\text{a}}$$\ddot{\text{a}}$n$\ddot{\text{a}}$nen and G. Xu \cite{Va-Gu}.
First, we introduce the result of Galochkin, Theorem $1$ in \cite{G1}. All notation and conventions are those of the above mentioned article, pages 385-388,  see also \cite{Nurma}.
\begin{theorem}[Galochkin, Theorem $1$ \cite{G1}]\label{Galochkin}
Let $I$ be $\Q$ or an imaginary quadratic field and $K$ be a finite extension of $I$ with $[K:I]=\kappa<\infty$.
For {$1\leq s\in \Z$}, consider $f_1(z),\ldots,f_s(z)\in K[[z]]$ which belong to the subclass $G(K,C_0,Q,\Lambda)$ with $C_0Q\ge 2$, $C=\max(1,C_0)$ $($see {\rm{\cite[Definition $1$, $2$]{G1}}}$)$. Assume that the functions are not connected by any non-zero polynomial in $s$ variables, of degree not exceeding $N$, with coefficients in $\C(z)$.
Let {$1\leq d\in\Z$} and $u:=\binom{N+s}{s}+\kappa\binom{N-d+s}{s}-\kappa\binom{N+s}{s}$ with $N\ge d$.\\
Suppose now
\begin{equation}
\label{existpade} u>0\,.
\end{equation} 
Then there exists an explicit constant $c_0>0$ which depends on $N,d$ and $f_1(z),\ldots,f_s(z)$, satisfying the following property$:$~for any integer with $|q|>c_0$ and a nonzero polynomial $P(x_1,\ldots,x_s)\in \Z[x_1,\ldots,x_s]$ of degree $d\le N$,
we have
$$P(f_1(1/q),\ldots,f_s(1/q))\neq 0\enspace.$$  
In particular, when $d=1$, we have $su=\binom{N-1+s}{N}\{N+s(1-\kappa)\}$.
Thus, under the condition that $N>s(\kappa-1)$ together with the assumption of the algebraic independence
of the functions $f_1(z),\ldots,f_s(z)$  over $\C(z)$, the linear independence of values of these $s$ functions over $K$
at the point $1/q$ follows. 
\end{theorem}
It is worth noting that Flicker \cite{Fli} proved a $p$-adic analogue of Galochkin's theorem. Building on both Galochkin's and Flicker's work, V$\ddot{\text{a}}$$\ddot{\text{a}}$n$\ddot{\text{a}}$nen  \cite{Va} refined the above mentioned results and generalized to a system of differential equations, both in the complex and the $p$-adic cases and also proved a Baker type lower bound for linear combinations of classical logarithms and polylogarithms, also subject to a metric condition as above.

\bigskip

For these results to work, one needs that the $G$-functions belong to the subclass $G(K,C_0,Q,\Lambda)$ with $C_0Q\ge 2$, that is, roughly speaking, 
a set of particular $G$-functions satisfying a system of linear differential equations, under the hypothesis
so-called Galochkin condition or $(G, C)$-condition, Definition $2$ in \cite{G1} ({\it same}  as  $(G, C)$-function condition in \cite{ch10} and as $(G, C)$-assumption in
\cite{Chubrothers}). 
 
More significant progress was made by Chudnovsky \cite{ch10}, who proved that, for $G$-functions satisfying a differential equation system as in (\ref{linear diff eq}), Galochkin's condition automatically holds.
 
 \vspace{\baselineskip}
Summing up, thanks to the above mentioned results, as soon as we can show that the considered $G$-functions satisfy a linear system of differential equations as in (\ref{linear diff eq}), as well as that the functions are {\it  linearly independent} over $\C(z)$, we get the linear independence of the special values {\it provided} condition~(\ref{existpade}) is satisfied.
Condition~(\ref{existpade}) comes from the use of Siegel's lemma to construct Pad\'e approximants (whereas we avoid using
Siegel's lemma in the present article).
 
We are now in a position to compare our results with the above mentioned series of results. Res\-training ourselves to 
the functions
$1, \Phi_s(x,\alpha_i z)$ with $1\le i \le m, 1 \le s \le r$, one can check they are {\it  linearly independent} over $\C(z)$,
in a similar way to \rm{\cite[p.\,\,292, 293]{Va}}; see  \cite{DHK1} (it may be worth noting that Galochkin's condition can be checked by hand in this special case, and thus one can also proceed without using Chudnovsky's observation). 
Hence, we are in the case $N=1$, thus necessarily $d=1$.
 
However, for $N=d=1$, condition~(\ref{existpade}) reads $u=s(1-\kappa)+1\le -s+1<0$ if $\kappa\ge 2$, hence the assumption $u>0$ of  Galochkin's theorem never holds when $N=1$ as soon as the base field considered is {\it not contained} in an imaginary quadratic field.

On the contrary, our criterion covers also such a case,  since the base field can be  an arbitrary number field.
Namely our result gives the linear independence of values of the Lerch functions, when $N=1$,
applying our explicit construction of Pad\'{e} approximations of $1,\Phi_s(x,\alpha_i/z)$
that is done around infinity, not around the origin (this is one of the reasons
why our corresponding assumption is much weaker than that of Galochkin's theorem).
Nevertheless, as we see in Example 6.3 below, our linear independence result for the values of Lerch functions is valid for algebraic points in $K$ 
{\it of arbitrary degree}, to which  neither  Galochkin's, nor V$\ddot{\text{a}}$$\ddot{\text{a}}$n$\ddot{\text{a}}$nen's results \cite{Va},   \cite{G1},  \cite{G2} apply.
\end{remark}

\begin{remark} It is also worth noting that our result (see \cite{DHK1} for details) is quantitative, with totally
explicit constants which is not the case of previous results.
\end{remark}

\begin{remark}\label{VAG}
A result by V$\ddot{\text{a}}$$\ddot{\text{a}}$n$\ddot{\text{a}}$nen-Xu \cite{Va-Gu} actually deals with general base fields as in our case. However, this is not applicable in our situation, because of the degenerate nature of the system~(\ref{linear diff eq}).
%
%
\end{remark}
%
%

\

The new ingredient in the article relies on a few points. First and foremost, we introduce a systematic construction of Pad\'e approximants, which heavily relies  on the computations made by past authors.
Our modifications and generalizations of the method of Niki\v{s}in developed in \cite{N0} \cite{N} as well as of the Rhin-Toffin method \cite{R-T},
supply {{\it a formally regulated  construction}}  of Pad\'e  approximants. Secondly an irrationality criterion, combined with the metric property provided for by Pad\'e approximation, leads to the irrationality for the values of the Lerch functions at
points sufficiently close to the origin (the precise sufficient condition, which we explain later, comes from the coupling of the criteria with Pad\'e approximation). This strategy works only if one can ensure {{\it the injectivity of evaluation maps}} defined by systems of Pad\'e approximation, which can be now interpreted as a non-vanishing property of a Hermite-type determinant,
which we succeed in proving. Our criterion also gives much more
relaxed assumptions than the previous results in \cite{G1}  \cite{G2}, since we
rely on our new {{\it formal}} construction of explicit Pad\'{e} approximants, by avoiding the use of Siegel's lemma.

\section{Notation and Main results}
We fix an algebraic closure of $\Q$ and denote it by $\overline{\Q}$. For a finite  subset $S\subset \overline{\Q}$, we define the denominator of $S$ by 
$${\rm{den}}(S):=\min\{0 < n\in\Z | \ n\alpha \ \text{is an algebraic integer for any} \ \alpha\in S\}\,.$$

\medskip

Let $\N$ be the set of strictly positive integers. Let $m, r\in \N$ and $K$ be an algebraic number field of finite degree over $\Q$. We denote the ring of integers of $K$ by $\mathcal{O}_K$ and the completion of $K$ with respect to the fixed embedding $\iota_{\infty}:K\hookrightarrow \C$ by $K_{\infty}$. 
Then $[K_{\infty}:\R]=1$ if $K_{\infty}\subset \mathbb{R}$, and
$[K_{\infty}:\R]=2$ otherwise.

\medskip

Let $x\in \Q\cap [0,1)$. 
Put $$\mu(x):={\rm{den}}(x)\prod_{q:{\rm{prime}}, q| {\rm{den}}(x)}q^{1/(q-1)}\enspace.$$
Consider 
$\boldsymbol{\alpha}:=(\alpha_1,\ldots,\alpha_m)\in (K\setminus\{0\})^m$ with $\alpha_i\neq \alpha_j$ for all $1\leq i< j \leq m$.
For $1\le g\le [K:\Q]$, we denote by  $\alpha^{(g)}$ the $g$-th conjugate 
of $\alpha\in K$ over $\Q$.  

\medskip

Let $\beta\in K\setminus\{0\}$ with $\max_{1\leq i \leq m} (|\alpha_i| )<|\beta|$. 
We put
$$D(\boldsymbol{\alpha},\beta):={\rm{den}}(\alpha_1,\ldots,\alpha_m,\beta)\enspace.$$ 

\noindent
We also define
\begin{align*} 
&\mathbb{A}(\boldsymbol{\alpha},\beta,x):={\rm{log\,}}|\beta|-(rm+1){\rm{log\,}}\max_i(|\alpha_i|)-\{rm({\rm{log\,}}D(\boldsymbol{\alpha},\beta)+r[{\rm{den}}(x)+{\rm{log\,}}(5/2)])+r({\rm{log\,}}3+{\rm{log\,}}\mu(x))\}\enspace,\\
&\mathcal{A}^{(g)}(\boldsymbol{\alpha},\beta,x):=rm\left({\rm{log\,}}D(\boldsymbol{\alpha},\beta)+{\rm{log\,}}\max(1, \min(|\alpha^{(g)}_i|)^{-1}\cdot|\beta^{(g)}|)+r[{\rm{den}}(x)-{\rm{log\,}}2]\right)\\
&\ \ \ \ \ \ \ \ \ \ +r\left({\rm{log\,}}\mu(x)+\sum_{i=1}^m{\rm{log\,}}(2^r|\alpha_i|+3^r\max(|\alpha^{(g)}_{i}|,|\beta^{(g)}|))\right)+{\rm{log\,}}3
 \ \ \ \ \text{for} \ \  1\le g\le [K:\Q]\enspace,
\end{align*}
and
$$
V(\boldsymbol{\alpha},\beta,x):=\mathbb{A}(\boldsymbol{\alpha},\beta,x)+\mathcal{A}^{(1)}(\boldsymbol{\alpha},\beta,x)-\dfrac{\sum_{g=1}^{[K:\Q]}\mathcal{A}^{(g)}(\boldsymbol{\alpha},\beta,x)}{[K_{\infty}:\R]}\enspace.
$$

\bigskip

We then obtain the following statement.
{{\begin{theorem} \label{Lerch}
Assume $V(\boldsymbol{\alpha},\beta,x)>0$.
Then the $rm+1$ numbers$:$~
$$1,\Phi_1(x,\alpha_1/\beta),\ldots,\Phi_r(x,\alpha_1/\beta),\ldots, \Phi_1(x,\alpha_m/\beta),\ldots,\Phi_r(x,\alpha_m/\beta)\enspace,$$ are linearly independent over $K$.
\end{theorem}}}
In the special case, where $K$ equals $\Q$ or an imaginary quadratic field,
Corollary $6$ in \cite{Va} gives an analogous quantitative result for polylogarithms, but the needed condition there is not so explicit as ours $V (\alpha,\beta,0)>0$. For a general number field $K$, Theorem $\ref{Lerch}$ is the first result to give the linear independence of the values of the Lerch function, even in the case of
polylogarithms, at distinct algebraic numbers.

\section{Construction of Pad\'{e} approximants}
We now explain how we construct Pad\'{e} approximants of the Lerch functions. Since the full proof is long, then with the relevant details, it will be provided for in the forthcoming articles \cite{DHK1},  \cite{DHK2},  with a $p$-adic analogue as well as quantitative measures of linear independence.

First we recall the definition of Pad\'{e} approximants of formal Laurent series. In the rest of this section, we denote by $L$ a unique factorization domain of characteristic $0$.  
We define the order function ${\rm{ord}}_{\infty}$ at ``$z=\infty$" by
\begin{align*}
{\rm{ord}}_{\infty}:L[z][[1/z]]\rightarrow \Z\cup \{\infty\} , \ \ \sum_{k}a_k\cdot\dfrac{1}{z^k}\mapsto \min\{k\in \Z\mid a_k\neq 0\}\enspace.
\end{align*}    
\begin{lemma} \label{pade}
Let $r$ be a positive integer, $f_1(z),\ldots,f_r(z)\in 1/z\cdot L[[1/z]]$ and $\bold{n}:=(n_1,\ldots,n_r)\in \N^{r}$.
Put $N:=\sum_{i=1}^rn_i$.
Let $M$ be a positive integer with $M\ge N$. Then there exists a family of polynomials 
$(P_0(z),P_{1}(z),\ldots,P_r(z))\in L[z]^{r+1}\setminus\{\bold{0}\}$ satisfying the following conditions$:$~
\begin{align*} 
&({\rm{i}}) \ {\rm{deg}}P_{0}(z)\le M\enspace,\\
&({\rm{ii}}) \ {\rm{ord}}_{\infty} P_{0}(z)f_j(z)-P_j(z)\ge n_j+1 \ \text{for} \ 1\le j \le r\enspace.
\end{align*}
\end{lemma}
\begin{definition} 
We use the same notation as those in Lemma $\ref{pade}$. 
We call a family of polynomials $(P_0(z),P_{1}(z),\ldots,P_r(z)) \in L[z]^{r+1}$ satisfying the properties $({\rm{i}})$ and $({\rm{ii}})$, 
{\it Pad\'{e} type approximants of} $(f_1,\ldots,f_r)$,  {\it  of weight} $\bold{n}$ {\it and of degree} $M$.

For the Pad\'{e} type approximants $(P_0(z),P_{1}(z),\ldots,P_r(z))$, of $(f_1,\ldots,f_r)$ of weight $\bold{n}$, 
we call the family of formal Laurent series $(P_{0}(z)f_j(z)-P_{j}(z))_{1\le j \le r}$,
{\it Pad\'{e} type approximation  systems of }$(f_1,\ldots,f_r)$,  {\it  of weight} $\bold{n}$ {\it and of degree} $M$.
\end{definition}

In the sequel, we take $x\in L\setminus \Z_{<0}$ and assume $x+k$ are invertible in $L$ for any $k\in \N$.

We now introduce notation for formal primitive, derivation, and evaluation maps. Let $I$ be a finite set, we assume that $L$ contains $K[X_i,1/X_i]_{X_i\in I}$ where $K$ is a number field. In the sequel, it will be convenient to work formally and thus to treat as many quantities as variables as is useful, and we shall freely extend the set $I$ as need arises.

\begin{notation} \label{notationderiprim}
\begin{itemize}
\item[(i)]  For $\alpha\in L$, We denote by ${\ev}_{\alpha}$ the linear evaluation map $L[{T}]\longrightarrow L$, $P\longmapsto P(\alpha)$. 
\item[(ii)] For $P\in L[T]$, we denote by $[P]$ the multiplication by $P$ ($Q\longmapsto PQ$).
\item[(iii)] We also denote by $\inte_x$ { (formal primitive) }the linear operator $L[T]\longrightarrow L[T]$, defined by $$P\longmapsto \frac{1}{T^{1+x}}\int_{0}^{T}{\xi}^xP(\xi)d\xi.$$ 
\item[(iv)] We denote by $\deri_x$ the derivative map $$P\longmapsto T^{-x}\tfrac{d}{dT}(T^{x+1}P(T)),$$ and for $n\geq 1$, by $S_{n,x}$ the map taking $$T^k {\longmapsto }\dfrac{(k+x+1)_n}{n!}T^k,$$ where 
$(k+x+1)_n:=(k+x+1)\ldots(k+x+n),$  that is, the
divided derivative mapping $$P\longmapsto \frac{1}{n!}T^{-x}\frac{d^n}{dT^n}(T^{n+x}P)=\dfrac{1}{n!}\left(\dfrac{d}{dT}+\dfrac{x}{T}\right)^nT^n(P),$$ so that $\deri_x=S_{1,x}$.
\item[(v)] If $\varphi$ is an $L$-automorphism of an $L$-module $M$ and $k$ an integer, we define
$$\varphi^{(k)}:=\begin{cases}
\overbrace{\varphi\circ\cdots\circ\varphi}^{k-\text{times}} & \ \text{if} \ k>0,\\
{\rm{id}}_M &  \ \text{if} \ k=0,\\
\overbrace{\varphi^{-1}\circ\cdots\circ\varphi^{-1}}^{-k-\text{times}} & \ \text{if} \ k<0.
\end{cases}
$$   
\end{itemize}
\end{notation}

For a given $l\in \Z$, we define the linear map $\varphi_{\alpha,x,l}$ as follows.
\begin{notation} 
$$\varphi_{\alpha,x,l}:=[\alpha]\circ\ev_{\alpha}\circ \inte^{(l)}_x\enspace.$$ 
\end{notation}

\

\

For any non-negative integers $k$, note that $\varphi_{\alpha,x,s}(T^k)$ is a formal analogue of 
$$\dfrac{1}{(s-1)!}\int^{\alpha}_{0}T^{k+x} {\rm{log\,}}^{s-1}\dfrac{1}{T}dT\enspace.$$

\

\

For convenience, we collect below the following elementary facts.

\

\begin{facts}
\label{faitselem}
\begin{itemize} 
\item[(i)] The map $\inte_x$ is an isomorphism and its inverse is $\deri_x$ for $x\in L\setminus\Z_{<0}$.
Hence $\varphi_{\alpha,x,s}$ is well-defined for $s\leq -1$.
\item[(ii)] For any integers $n_1\geq 0, n_2\geq 0$ and $x\in L\setminus\Z_{<0}$ with $x+k$ invertible in $L$ for any $k\in \N$, the divided derivatives $S_{n_1,x}$ and $S_{n_2,x}$ commute, namely $S_{n_1,x}\circ S_{n_2,x}=S_{n_2,x}\circ S_{n_1,x}$.
\item[(iii)]  For any integer $s\in\zu$ and any $\alpha\in L$, we have $\varphi_{\alpha,x,s}\circ\deri_x=\varphi_{\alpha,x,s-1}$.
\item[(iv)] By continuity, all the above mentioned maps extend to $L[[T]]$ with respect to the natural valuation.
\item[(v)] The kernel of the map $\varphi_{\alpha,x,0}$ is the ideal $(T-\alpha)$ for any $x\in L\setminus\Z_{<0}$.
\end{itemize}
\end{facts}

\

Using Fact~\ref{faitselem} (iv), the classical Lerch function is indeed expressed as a natural image by $\varphi_{\alpha,x,s}$ with $s\geq 1$ by
\begin{align} \label{Lerch integral rep}
\varphi_{\alpha,x,s}\left(\dfrac{1}{z-T}\right)=\Phi_s(x,\alpha/z)\enspace.
\end{align}

\

Consider
$\boldsymbol{\alpha}:=(\alpha_1,\ldots,\alpha_m)\in (L\setminus\{0\})^m$ with $\alpha_i\neq \alpha_j$ for $i\neq j$.
We study Pad\'{e} approximants of type  I\hspace{-.1em}I of the functions $(\Phi_{s}(x,\alpha_i/z))_{\substack{1\le i \le m \\ 1\le s \le r}}$.

Let $l$ be a non-negative integer with $0\le l \le rm$. For  a positive integer $n$, we define the family of polynomials:~
\begin{align}
&P_{n,l}(\boldsymbol{\alpha},x|z):=\ev_z\circ S^{(r)}_{n,x}\left(T^l\prod_{i=1}^m(T-\alpha_i)^{rn}\right), \label{Qnl}\\
&P_{n,l,i,s}(\boldsymbol{\alpha},x|z):=\varphi_{\alpha_i,x,s}\left(\dfrac{P_{n,l}(\boldsymbol{\alpha},x|z)-P_{n,l}(\boldsymbol{\alpha},x|T)}{z-T}\right) \ \text{for} \ 1\le i \le m, 1\le s \le r\enspace.\label{Qnlijsj}
\end{align}

\bigskip

Under the notation above, we obtain the following theorem.

\

\begin{theorem} \label{Pade appro Lerch} 
For each $0\leq l\leq rm$, the family of polynomials $(P_{n,l}(\boldsymbol{\alpha},x|z),P_{n,l,i,s}(\boldsymbol{\alpha},x|z))_{\substack{1\le i \le m \\ 1\le s \le r}}$ forms  a Pad\'{e} type approximants system of $(\Phi_{s}(x,\alpha_i/z))_{\substack{1\le i \le m \\ 1\le s \le r}}$,
of
weight $(n,\ldots,n)\in \N^{rm}$ and of degree $rmn+l$.

\end{theorem} 
\begin{proof}
By the definition of $P_{n,l}(\boldsymbol{\alpha},x|z)$, we have
$$
{\rm{deg}}P_{n,l}(\boldsymbol{\alpha},x|z)=rmn+l\,.
$$
Hence the condition on the  degree is verified. We only need  to check the condition on the valuation.

Put $$R_{n,l,i,s}(\boldsymbol{\alpha},x|z):=P_{n,l}(\boldsymbol{\alpha},x|z)\Phi_{s}(x,\alpha_i/z)-P_{n,l,i,s}(\boldsymbol{\alpha},x|z).$$
Then, by  definition of $R_{n,l,i,s}(\boldsymbol{\alpha},x|z)$ with the property $(\ref{Lerch integral rep})$, we obtain

\begin{align}\label{error}
R_{n,l,i,s}(\boldsymbol{\alpha},x|z)&=\displaystyle P_{n,l}(\boldsymbol{\alpha},x|z)\varphi_{\alpha_i,x,s}\left(\dfrac{1}{z-T}\right)-P_{n,l,i,s}(\boldsymbol{\alpha},x|z) \nonumber \\
                                                       &=\displaystyle\varphi_{\alpha_i,x,s}\left(\dfrac{P_{n,l}(\boldsymbol{\alpha},x|T)}{z-T}\right)
                                                       =\sum_{k=0}^{\infty}\varphi_{\alpha_i,x,s}(T^kP_{n,l}(\boldsymbol{\alpha},x|T))\dfrac{1}{z^{k+1}}\enspace.\qquad\qquad  {(7)}
\end{align}

Note that in ${\rm{End}}_K(K[T])$ we have the following identities 
\begin{align*}
&S_{n,x}=\dfrac{1}{n!}S_{1,x}\circ\ldots \circ (S_{1,x}+n-1) \ \text{for} \ n\in \N\enspace,\\
&[T^k]\circ S_{1,x}=(S_{1,x}-k)\circ [T^k] \ \text{for} \ k\in \Z_{\ge0}\enspace.
\end{align*}
By the definition of $P_{n,l}(\boldsymbol{\alpha},x|T)$ and the  identities above, for each $1\le s \le r, \,\,0\le k \le n-1$, there exists a polynomial $U_{s,k}(X)\in \Q[X]$ of\,\, ${\rm{deg}}U_{s,k}=nr-s$, satisfying
$$T^kP_{n,l}(\boldsymbol{\alpha},x|T)=S^{(s)}_{1,x}\circ U_{s,k}(S_{1,x})\left(T^{k+l}\prod_{i=1}^m(T-\alpha_i)^{rn}\right).$$
By the Leibniz rule, we obtain that  $U_{s,k}(S_{1,x})\left(T^{k+l}\prod_{i=1}^m(T-\alpha_i)^{rn}\right)$
belongs to the ideal $ (T-\alpha_i)$ for each $1\le i \le m$.
Hence we get
\begin{align*}
\varphi_{\alpha_i,x,s}(T^kP_{n,l}(\boldsymbol{\alpha},x|T))
                                                                          &=\varphi_{\alpha_i,x,0}\circ U_{s,k}(S_{1,x})\left(T^{k+l}\prod_{i=1}^m(T-\alpha_i)^{rn}\right)=0\enspace,
\end{align*}
for $1\le i \le m, 1\le s \le r$ and $0\le k \le n-1$.

Consequently, by the expansion above of $R_{n,l,i,s}(\boldsymbol{\alpha},x|z)$, we obtain
$$
{\rm{ord}}_{\infty}R_{n,l,i,s}(\boldsymbol{\alpha},x|z)\ge n+1 \ \text{for} \ 1\le i \le m, 1\le s \le r\enspace.
$$
Then Theorem $\ref{Pade appro Lerch}$ follows.
\end{proof}

\section{Metric approximations and linear independence criteria}

We now give a few of the estimates associated with the Pad\'e approximation we just constructed. They do not need involved arguments to be proven; however due to the technical nature of the construction, computations are somewhat heavy and we skip them to keep in line with the spirit of this article. 

The estimates in Lemma \ref{upper coefficients}
can be combined with an appropriate linear independence criterion to provide for a measure.
\begin{lemma} \label{upper coefficients}
Let $n$ be a positive integer, $x$ a rational number with $0\le x<1$ and $\beta\in K\setminus\{0\}$. Then for any $1\le g \le [K:\Q]$, we have

\begin{align*}
\max_{\substack{0\le l \le rm \\ 1\le i \le m \\ 1\le s \le r}}|P^{(g)}_{n,l,i,s}(\boldsymbol{\alpha},x|\beta)|&\le
\max(|\alpha^{(g)}_i|)^{rm}\left(\dfrac{3}{2}\right)^{r^2m+r}\left(\dfrac{3}{2^{rm}}\prod_{j=1}^m\left[2^r|\alpha^{(g)}_{j
}|+3^r\max(|\alpha^{(g)}_{i}|)\right]\right)^{rn} \\
&\times\begin{cases}
\dfrac{\left(\min(|\alpha^{(g)}_i|)^{-1}|\beta^{(g)}|\right)^{rm(n+1)}}{\min(|\alpha^{(g)}_i|)^{-1}|\beta^{(g)}|-1} & \ \text{if} \ \min(|\alpha^{(g)}_i|)^{-1}|\beta^{(g)}|>1\enspace\\
rm(n+1)  & \ \text{if} \ \min(|\alpha^{(g)}_i|)^{-1}|\beta^{(g)}|\le 1\enspace,
\end{cases}
\end{align*}
for $1\le i \le m$.

For the error term, we have$:$~
\begin{align*}
\max_{0\le l \le rm}|R_{n,l,i,s}(\boldsymbol{\alpha},x|\beta)|\le \dfrac{\max_{1\le i \le m}(1,|\alpha_i|)^{rm+1}}{|\beta|-\max_{
j
}(|\alpha_{
j
}|)}\left(\dfrac{3}{2}\right)^{r^2m+r}\left(\dfrac{\max_{
j
}(|\alpha_{
j
}|)^{rm+1}}{|\beta|}\right)^n
 \left(3\left(\dfrac{5}{2}\right)^{rm}\right)^{rn}\enspace.
\end{align*}
\end{lemma}
We give here an outline of the proof. 
By $(\ref{Qnl})$ and $(\ref{Qnlijsj})$, we have
\begin{align*}
&P_{n,l}(\boldsymbol{\alpha},x|z)=\sum_{k=0}^{rmn} \left[\sum_{\substack{ 1\le i \le m\\ 0\le k_i \le rn \\ \sum_{i}k_i=k}}\left(\prod_{i=1}^m\binom{rn}{k_i}(-\alpha_i)^{rn-k_i}\right)\right]\left(\dfrac{(k+l+x+1)_n}{n!}\right)^{r}z^{k+l},\\
&P_{n,l,i,s}(\boldsymbol{\alpha},x|z)=\sum_{u=\max(l,1)-1}^{rmn+l-1}\left[\sum_{k=u+1}^{rmn+l}\left(\sum_{\substack{ 1\le {i^{\prime} }\le m\\ 0\le {k_{i^{\prime}}} \le rn \\ {\sum_{i^{\prime}}k_{i^{\prime}}}=k-l}}{\prod_{i^{\prime}=1}^m}{\binom{rn}{k_{i^{\prime}}}}(-{\alpha_{i^{\prime}}})^{rn-{k_{i^{\prime}}}}\right) \times\left(\dfrac{(1+k+x)_n}{n!}\right)^{r}\dfrac{\alpha^{k-u}_i}{(k-u)^{s}}\right]z^{u}.
\end{align*}
By the above equalities together with the triangle inequality, we obtain the upper bound for 
$|P^{(g)}_{n,l,i,s}(\boldsymbol{\alpha},x|\beta)|$  
{and $|P_{n,l}(\boldsymbol{\alpha},x|\beta)|$}.
For the term $|R_{n,l,i,s}(\boldsymbol{\alpha},x|\beta)|$, we use {$(\ref{error})$.  \sout{for}}

\bigskip

We then state a general linear independence criterion:~
\begin{proposition} \label{critere version II} 
Let $K$ be an algebraic number field of finite degree over $\Q$. We denote the completion of $K$ with respect to the fixed embedding $\iota_{\infty}$ by $K_{\infty}$.
Let $m\in\N$ and $\theta_0:=1,\theta_1,\ldots,\theta_{m}\in \C\setminus\{0\}$. 
\\
Suppose that there exists a set of matrices $$\{(A_{n,l,j})_{0\le l,j\le m}\}_{n\in \N}\subset {\rm{M}}_{m+1}(\mathcal{O}_K)\cap {\rm{GL}}_{m+1}(K)\,.$$ 
Assume further that there exist
positive real numbers
$$
\{\mathcal{A}^{(g)}\}_{\substack{1\le g \le [K:\Q]}},$$
and a positive real number $\mathbb{A},
$
satisfying the conditions$:$~

\begin{align}\label{upper Anrj}
\max_{0\le l,j \le m}|A^{(g)}_{n,l,j}|\le e^{\mathcal{A}^{(g)}\cdot n+o(n)} \ \text{for} \ 1 \le g \le [K:\Q]\  \  \
(n\rightarrow \infty),{\qquad (8)}
\end{align} 

\begin{align}\label{upper Rjn}
\max_{\substack{0\le l \le m \\ 1\le j \le m}} |A_{n,l,0}\cdot \theta_j-A_{n,l,j}|\le e^{-\mathbb{A}\cdot n+o(n)}\  \  \
(n\rightarrow \infty).\qquad{\qquad (9)}
\end{align}
We put 
$$
V:=\mathbb{A}+\mathcal{A}^{(1)}-\dfrac{\sum_{g=1}^{[K:\Q]}\mathcal{A}^{(g)}}{[K_{\infty}:\R]}\,.
$$
If  $V>0$, then the numbers ${\theta_0},\ldots,\theta_m$  are linearly independent over $K$. 
\end{proposition}
\begin{proof} 
Assume that there exists a vector $\boldsymbol{\beta}:=(\beta_0,\ldots,\beta_m)\in \mathcal{O}_K\setminus \{\bold{0} \}$ 
satisfying $\Lambda(\boldsymbol{\beta},\boldsymbol{\theta}):={\displaystyle\sum_{i=0}^m}\beta_i\theta_i=0$.
For $n\in\N$, as we have ${\rm{det}}(A_{n,l,j})_{0\le l,j \le m}\neq 0$, there exists $0\le l_n \le m$ satisfying 
\begin{align*} 
B_{l_n}:=\sum_{j=0}^{m}A_{n,l_n,j}\beta_j\neq 0.
\end{align*}
Put $R_{n,l,j}=A_{n,l,0}\theta_j-A_{n,l,j}$ for $1\le j \le m$ and $0\le l \le m$.
Then by the definitions of $\Lambda(\boldsymbol{\beta}, \boldsymbol{\theta})$, $B_{l_n}$, and $R_{n,l,j}$, we obtain
$$ 
0=A_{n,l_n,0}\Lambda(\boldsymbol{\beta}, \boldsymbol{\theta})=B_{l_n}+\sum_{j=1}^{m}R_{n,l_n,j}\beta_j.
$$
Using the product formula for $B_{l_n}\in \mathcal{O}_K\setminus\{0\}$, it follows that
\begin{align} \label{upper infty}
1&\le {{\prod_{g}}^{\prime}} |B^{(g)}_{l_n} | \times |B_{l_n}|^{[K_{\infty}:\R]}=
{{\prod_{g}}^{\prime}} |B^{(g)}_{l_n} | \times \left|\sum_{j=1}^{m}R_{n,l_n,j}\beta_j\right|^{[K_{\infty}:\R]}.
\end{align}
Here ``$ \ {}^{\prime} \ $'' in {${\prod_{g}}^{\prime}$},
$g$ runs $2\le g \le [K:\Q]$ if $K_{\infty}=\R$ and $3\le g \le [K:\Q]$ if $K_{\infty}=\C$.
Firstly, we look for an upper bound of $|B^{(g)}_{l_n}|$ for $g\neq 1$ if $K_{\infty}=\R$ and $g\neq 1,2$ if $K_{\infty}=\C$. 
\\
Using inequality {$(\ref{upper Anrj})$}, we have 
\begin{align} \label{upper Brn}
|B^{(g)}_{l_n}|\le e^{\mathcal{A}^{(g)}n+o(n)} \ (n\to\infty). 
\end{align}
Secondly, we give an upper bound for $\left|\sum_{j=1}^{m}R_{n,l_n,j}\beta_j\right|$. 
By {$(\ref{upper Rjn})$}, we get
\begin{align} \label{upper Brn iota}
\left|\sum_{j=1}^{m}R_{n,l_n,j}\beta_j\right|&\le e^{-\mathbb{A}n+o(n)} \ (n\to \infty).
\end{align}
Substituting the inequalities $(\ref{upper Brn})$ and $(\ref{upper Brn iota})$ into inequality $(\ref{upper infty})$,
taking the $1/[K_{\infty}:\R]$-th power of the inequality, we obtain
\begin{align*} 
1&\le e^{-Vn+o(n)} \ (n\to \infty). 
\end{align*}
Since $V>0$, we arrive at a contradiction for this inequality for all sufficiently large $n\in \N$.
\end{proof}

\

Theorem~\ref{Pade appro Lerch} gives us the sequence of  matrices. The growth control of the size of the matrices to carry out the approximations is provided for in Lemma~\ref{upper coefficients}.  However, the matrices do
not always
have algebraic integer entries. This is not a big deal. The defect of integrality comes from our operators $\inte_x,\deri_x$ and it  is corrected by multiplying by a suitable power of the least common multiple $d_n:={\rm l.\,c.\,m.}(1,\ldots,n)$ which is standard in the theory.

Plugging in these estimates in Proposition \ref{critere version II} leads us to the proof of the
main theorem. The metric condition requiring the numbers to be sufficiently close to the origin, is translated to the condition $V>0$ in {the linear independence criterion (Proposition \ref{critere version II}).} 

However, there is still a significant step to be performed. Now we need to prove that the matrices coming from the Pad\'e  approximation are indeed invertible. We describe this main step in the next section.

\section{Non-vanishing of a determinant and the final step of the proof}
In this section, we use the following notation.
Let $m, r$ be positive integers and $K$ be a field with characteristic $0$. We assume that $\alpha_1,\ldots,\alpha_m, z,T$ all belong to the set of variables $I$, so our ring $L$ contains $K[\alpha_i,z,T,1/\alpha_i,1/z,1/T]$. Put  $\boldsymbol{\alpha}:=(\alpha_1,\ldots,\alpha_m)$.

For a positive integer $l$ with $0\le l \le rm$, and for $x\in K$, we put
\begin{align*}
&P_{n,l}(z):=P_{n,l}(\boldsymbol{\alpha},x|z)\enspace,\\
&P_{n,l,i,s}(z):=P_{n,l,i,s}(\boldsymbol{\alpha},x|z) \ \text{for} \ 1\le i \le m, 1\le s \le r\enspace.
\end{align*}
The polynomials in the right-hand sides above have been already defined in $(\ref{Qnl})$ and $(\ref{Qnlijsj})$ respectively.

We define a column vector $\vec{p}_{n,l}(z)\in K[z]^{rm+1}$ by
\begin{align*}
&\vec{p}_{n,l}(z):={}^t\Biggl(P_{n,l}(z),
{P_{n,l,1,1}(z),\ldots, P_{n,l,1,r}(z)}, \ldots, {P_{n,l,m,1}(z),\ldots, P_{n,l,m,r}(z)}\Biggr).
\end{align*}

\

\begin{proposition} \label{non zero det}
We use the same notation as above. For any positive integer $n$, we have
$$
\Delta_n(z):=
{\rm{det}}  {\begin{pmatrix}
\vec{p}_{n,0}(z) \ \cdots \ \vec{p}_{n,rm}(z)
\end{pmatrix}}
\in K{{(\alpha_1,\alpha_2,\ldots,\alpha_m)}}\setminus\{0\}\enspace. 
$$
\end{proposition}

\

To prove this, we firstly prove that the determinant $\Delta_n=\Delta_n(z)$ is a constant,  i.~e.~is independent of $z$.
Secondly, we regard $\Delta_n$ as an element of $K(\alpha_1,\ldots,\alpha_m)$ viewing  $\alpha_1,\ldots,\alpha_m$ as
indeterminates, and factor it up to a constant depending only on $n, m, r$. We finally show that this absolute constant
$\Delta_n$ is non-zero.
For this last step, we  identify this determinant with a certain {real} integral  to show that it does not vanish.

\

We shall prove:
$$
\Delta_n(z)\in K{{(\alpha_1,\ldots,\alpha_m)}} \ \text{for all} \ n\in \N\enspace. 
$$

\

For this, denote $P_{n,l}(z)\Phi_{s}(x,\alpha_i/z)-P_{n,l,i,s}(z)$ by $R_{n,l,i,s}(z)$  as above ($0\le l \le rm$, $1\le i \le m$, $1\le s \le r$).

In the matrix giving the determinant $\Delta_n(z)$, we add, the first row multiplied by the $\Phi_{s}(x,\alpha_i/z)$, 
to the $(i-1)r+s+1$-th row ($1\le i \le m$, $1\le s \le r$), to obtain
                     $$ 
                     \Delta_n(z)=(-1)^{rm}{\rm{det}}
                     {\begin{pmatrix}
                     P_{n,0}(z) & \dots &P_{n,rm}(z)\\
                     R_{n,0,1,1}(z) & \dots & R_{n,rm,1,1}(z)\\
                     \vdots    & \ddots & \vdots  \\
                     R_{n,0,1,r}(z) & \dots & R_{n,rm,1,r}(z)\\
                     \vdots & \ddots & \vdots\\
                     R_{n,0,m,1}(z) & \dots & R_{n,rm,m,1}(z)\\
                     \vdots    & \ddots & \vdots  \\
                     R_{n,0,m,r}(z) & \dots & R_{n,rm,m,r}(z)\\
                     \end{pmatrix}}\enspace. 
                     $$ 
We denote   by $\Delta_{n,s,t}(z)$, the $(s,t)$-th cofactor of the matrix in the right-hand side of the identity above.
Then we have, developing along the first row
\begin{align} \label{formal power series rep delta}
\Delta_n(z)=(-1)^{rm}\left(\sum_{l=0}^{rm}P_{n,l}(z)\Delta_{n,1,l+1}(z)\right)\enspace.
\end{align} 
Since we have 
\begin{align*}
{\rm{ord}}_{\infty} R_{n,l,i,s}(z)\ge n+1 \ \text{for} \ 0\le l \le rm, \ 1\le i\le m \ \text{and} \ 1\le s \le r\enspace,
\end{align*}
we get
$$
{\rm{ord}}_{\infty}\Delta_{n,1,l+1}(z)\ge (n+1)rm\enspace.
$$
Combining the fact ${\rm{deg}}P_{n,l}(z)=rmn+l$ with the lower bound of ${\rm{ord}}_{\infty}\Delta_{n,1,l+1}(z)$ above, we obtain
$$
P_{n,l}(z)\Delta_{n,1,l+1}(z)\in 1/z\cdot K[[1/z]] \ \text{for} \ 0\le l \le rm-1\,,
$$
and
$$
P_{n,rm}(z)\Delta_{n,1,rm+1}(z)\in K[[1/z]].
$$
Note that in the  relation above, the constant term of $P_{n,rm}(z)\Delta_{n,1,rm+1}(z)$ is 
$$
``\text{Coefficient of} \ z^{rm(n+1)} \ \text{of} \ P_{n,rm}(z)"\times ``\text{Coefficient of} \ 1/z^{rm(n+1)} \ \text{of} \ \Delta_{n,1,rm+1}(z)".
$$ 
Thus by $(\ref{formal power series rep delta})$, the determinant $\Delta_n(z)$ is a polynomial in $z$ with non-positive valuation with respect to ${\rm{ord}}_{\infty}$, consequently it turns to be a constant. 
Moreover, the terms of  strictly negative valuation should be canceled out. Hence we have
\begin{align}\label{cal delta n}
\Delta_n=\Delta_n(z)&=(-1)^{rm}\kern-3pt\times\kern-4pt\left(\sum_{l=0}^{rm}P_{n,l}(z)\Delta_{n,1,l+1}(z)\right)\nonumber\\
                           &=(-1)^{rm}\kern-3pt\times\kern-3pt ``\text{constant term of}\,  P_{n,rm}(z)\Delta_{n,1,rm+1}(z)" \in K.\enspace
\end{align}
%

\vspace{\baselineskip}

We now need to rewrite $\Delta_n$ as a rational function of $\alpha_1,\ldots,\alpha_m$ in a workable way.
We further extend the set of variables and assume that the set $I$ contains the $rm$ variables $t_{i,s}, 1\leq i\leq m, 1\leq s\leq r$, so that $L$ contains 
$$K[\alpha_1,\ldots,\alpha_m, z,T,1/\alpha_1,\ldots,1/\alpha_m,1/z,1/T][t_{i,s}].$$

For each variable $t_{i,s}$ and any integer $l$, we have a well-defined map for $\alpha\in L$:
\begin{equation*}
\varphi_{\alpha,t_{i,s}, x,l}:L[t_{i,s}]_{1\leq i\leq m,1\leq s\leq r}\longrightarrow L[{t_{i^{\prime},s^{\prime}}]_{(i^{\prime},s^{\prime})\neq (i,s)}}, \ \ t^k_{i,s}\mapsto 
\dfrac{{\alpha^{k+1}}}{(k+x+1{)^{{l}}}}.
\end{equation*}
Since  $L[t_{i,s}]_{1\leq i\leq m,1\leq s\leq r}$ can be regarded as a polynomial ring in one variable 
${L^{\prime}}[t_{i,s}]$ over ${L^{\prime}}
=L[{t_{i^{\prime},s^{\prime}}]_{(i^{\prime},s^{\prime})\neq (i,s)}}$.
  

\

Now for a positive integer $n$ and an integer $l$ with $ 0\le l \le rm$, we put 
$$
A_{n,l}(T):=T^l\prod_{i=1}^m(T-\alpha_i)^{rn}\enspace.
$$
By the definition of $A_{n,l}(T)$, we have $P_{n,l}(z)=\ev_z\circ S^{(r)}_{n,x}(A_{n,l}(T)).$

\ 

Let us define a column vector $\vec{r}_{n,l}\in L^{rm}$ by
{\footnotesize{
\begin{align*}
&\vec{r}_{n,l}:=\\
&{}^t\Biggl(\varphi_{\alpha_1,t_{1,1}, x,1}(t^n_{1,1}A_{n,l}(t_{1,1})), \ldots, \varphi_{\alpha_1,t_{1,r}, x, r}(t^n_{1,r}A_{n,l}(t_{1,r})), \ldots, \varphi_{\alpha_m,t_{m,1}, x,1}(t^n_{m,1}A_{n,l}(t_{m,1})), \ldots, \varphi_{\alpha_m,t_{m,r}, x, r}(t^n_{m,r}A_{n,l}(t_{m,r}))\Biggr)\enspace.
\end{align*}}}
\begin{lemma} \label{another presentation}
Under the notation above, we obtain the identity:~
$$ 
\Delta_n=(-1)^{rmn}{\left(\dfrac{(1+rmn+rm+x)_n}{n!}\right)^r }
{\rm{det}}  {\begin{pmatrix}
\vec{r}_{n,0} \
\cdots \
\vec{r}_{n,rm-1}
\end{pmatrix}}\enspace.
$$ 
\end{lemma}
\begin{proof}
Using $(\ref{cal delta n})$, we calculate constant term of $P_{n,rm}(z)\Delta_{n,1,rm+1}(z)\in K[[1/z]]$.

We need to deal with the non-commutativity of the multiplication by $[T]$ and the morphisms $S_{n,x}^{(k)}$. The defect of 
the commutativity is given by the following identity:~
there exists a set of rational numbers $\{e_{n,k}\}_{0\le k \le rn}\subset \qu$ with $e_{n,0}=(-1)^{rn}$ and
$$ 
[T^n]\circ S^{(r)}_{n,x}=\sum_{k=0}^{rn}e_{n,k}S^{(k)}_{1,x}\circ [T^n]\enspace.
$$
Then we obtain  
\begin{align*}
\varphi_{\alpha_i,x,s}(T^nP_{n,l}(T))&=\sum_{k=0}^{rn}e_{n,k}\varphi_{\alpha_i,x,s}\circ S^{(k)}_{1,x}\circ[T^n](A_{n,l}(T))\\ 
                                         &=\sum_{k=0}^{s-1}e_{n,k}\varphi_{\alpha_{i,x,s-k}}\circ[T^n](A_{n,l}(T))+\sum_{k=s}^{rn}e_{n,k}\varphi_{\alpha_i,x,0}\circ S^{(k-s)}_{1,x}\circ[T^n](A_{n,l}(T))\nonumber\\
                                         &=\sum_{k=0}^{s-1}e_{n,k}\varphi_{\alpha_{i,x,s-k}}(T^nA_{n,l}(T)), 
\end{align*}
for $1\le i \le m$ and $1\le s \le r$, the conclusion follows, interpreting the above relations as linear manipulations of lines
and columns leaving the determinant unchanged. 
\end{proof}

Now, for non-negative integers $u, n$, we put:
$$P_{u,n}(t_{i,s})=\prod_{i=1}^{m}\prod_{s=1}^r\left[ t_{i,s}^u\prod_{j=1}^m(t_{i,s}-\alpha_j)^{rn}\right]\prod_{(i_1,s_1)<(i_2,s_2)}(t_{i_2,s_2}-t_{i_1,s_1}),$$
where the order $(i_1,s_{1})<(i_2,s_{2})$ follows lexicographically.

\

By $\bigcirc$, we denote the composite of morphisms.
When no confusion is deemed to occur, we omit the subscripts $\underline{\alpha}=(\alpha_1,\ldots,\alpha_m)$ and write
$$\psi=\psi_{\underline{\alpha}}:=\bigcirc_{i=1}^{m}\bigcirc_{s=1}^r\varphi_{\alpha_i,t_{i,s},x,s}\enspace.$$

Note that, by definition of 
$
{\rm{det}}  {\begin{pmatrix}
\vec{r}_{n,0} \
\cdots \
\vec{r}_{n,rm-1}
\end{pmatrix}},
$
we have
$$
{\rm{det}}  {\begin{pmatrix}
\vec{r}_{n,0} \
\cdots \
\vec{r}_{n,rm-1}
\end{pmatrix}}=\psi(P_{n,n})\enspace.
$$
Let $u$ be a non-negative integer. We are going to study the value
$$
C_{n,u,m}:=\psi(P_{u,n})\,.
$$
By induction, we obtain the following proposition.
\begin{proposition} \label{decompose Cnum}
There exists a non-zero constant $c_{n,u,m}\in K$  satisfying
$$
C_{n,u,m}=c_{n,u,m}\prod_{i=1}^m\alpha^{r(u+1)+r^2n+\binom{r}{2}}_i \prod_{1\le i_1<i_2\le m}(\alpha_{i_2}-\alpha_{i_1})^{(2n+1)r^2},
$$
with $\displaystyle\binom{r}{2}=0$ if $r=1$. 
\end{proposition}

\

We write the details of the proof of the proposition in the forthcoming articles \cite{DHK1} \cite{DHK2}, however, we describe here our basic idea.
Indeed, we prove the proposition by {reducing to the case $m=2$ and }showing:
\begin{itemize}
\item[(i)] $C_{n,u,2}$ is homogeneous of degree $2r(u+1)+2r^2n+2{r\choose 2}+(2n+1)r^2$.
\item[(ii)] $(\alpha_1\alpha_2)^{r(u+1)+r^2n+{r\choose 2}}$ divides $C_{n,u,2}$.
\item[(iii)] $(\alpha_1-\alpha_2)^{(2n+1)r^2}$ divides $C_{n,u,2}$.
\end{itemize}

\

Here, we explain how the constant $c_{n,u,m}$ in Proposition $\ref{decompose Cnum}$ becomes non-zero. 
Whenever it is shown, then the determinant does not vanish.
\

We use the same notation as those in Proposition $\ref{decompose Cnum}$.
Define $$D_{n,u,m}:=\dfrac{C_{n,u,m}}{\prod_{i=1}^m\alpha^{r(u+1)+r^2n+\binom{r}{2}}_i}=c_{n,u,m}\times\prod_{1\le i_1<i_2\le m}(\alpha_{i_2}-\alpha_{i_1})^{(2n+1)r^2}.$$ 
A straightforward calculation of an integral gives us
\begin{align*}
D_{n,u,m}=&{\bigcirc_{i^{\prime}=1}^{m}\bigcirc_{s^{\prime}=1}^r\,\,\varphi_{1,{t_{i^{\prime},s^{\prime}}},x,s^{\prime}}}
\Biggl(\prod_{i=1}^m\prod_{s=1}^{r} \left[t^{u}_{i,s} \cdot (t_{i,s}-1)^{rn} \cdot \prod_{\substack{\tilde{i}\neq i\\ 1\le \tilde{i} \le m}}(\alpha_i t_{i,s}-\alpha_{\tilde{i}})^{rn}\right]\\
&\times\prod_{i=1}^m\left(\prod_{s_{1}<s_{2}}(t_{i,s_{2}}-t_{i,s_{1}})\right)\times
{\prod_{i_1<i_2}\kern3pt\prod_{\substack{1\leq
s_1,\,s_2\leq r}}}
(\alpha_{i_2}t_{i_2,s_{2}}-\alpha_{i_1}t_{i_1,s_{1}})\Biggr).
\end{align*}
We substitute $\alpha_m=0$ in $D_{n,u,m}$, then we have
{\small{
\begin{align*}
&\left. D_{n,u,m}\right|_{\alpha_m=0}=c_{n,u,m}\prod_{i=1}^{m-1}(-\alpha_i)^{(2n+1)r^2}\prod_{1\le i_1<i_2\le m-1}(\alpha_{i_2}-\alpha_{i_1})^{(2n+1)r^2}\\
&=\pm\prod_{i=1}^{m-1}\alpha^{(2n+1)r^2}_i {\bigcirc_{s^{\prime}=1}^{r}\,\,\varphi_{1, {t_{m,s^{\prime}}},x,s^{\prime}}}\left({\prod_{s=1}^{r}} \left[t^{u}_{m,s} \cdot (t_{m,s}-1)^{rn}\right] \times{\prod_{\substack{1\leq
s_1<s_2\leq r}}}
(t_{m,s_{2}}-t_{m,s_{1}})\right) \\
&\times {\bigcirc_{i^{\prime}=1}^{m-1}\bigcirc_{s^{\prime}=1}^{r}\,\,\varphi_{1, {t_{i^{\prime},s^{\prime}}},x,s^{\prime}}}\Biggl(\prod_{i=1}^{m-1}{\prod_{s=1}^{r}} \left[t^{u+r(n+1)}_{i,s} \cdot (t_{i,s}-1)^{rn} \cdot \prod_{\substack{\tilde{i}\neq i\\ 1\le \tilde{i} \le m-1}}(\alpha_i t_{i,s}-\alpha_{\tilde{i}})^{rn}\right]\\
&\times\prod_{i=1}^{m-1}\left({\prod_{\substack{1\leq
s_1<s_2\leq r}}}
({t_{i,s_{2}}}-t_{i,s_{1}})\right)\times 
{\prod_{1\le i_1<i_2\le m-1}\kern3pt\prod_{\substack{1\leq
s_1,\,s_2\leq r}}}
(\alpha_{i_2}t_{i_2,s_{2}}-\alpha_{i_1}t_{i_1,s_{1}})\Biggr)\\
&=\pm\prod_{i=1}^{m-1}\alpha^{(2n+1)r^2}_i {\bigcirc_{s^{\prime}=1}^{r}\,\,\varphi_{1, {t_{m,s^{\prime}}},x,s^{\prime}}}\left({\prod_{s=1}^{r}} \left[t^{u}_{m,s} \cdot (t_{m,s}-1)^{rn}\right] \times{\prod_{\substack{1\leq
s_1<s_2\leq r}}}(t_{m,s_{2}}-t_{m,s_{1}})\right) \\
&\times c_{n,u+r(n+1),m-1}\prod_{1\le i_1<i_2\le m-1}(\alpha_{i_2}-\alpha_{i_1})^{(2n+1)r^2}.
\end{align*}}}
Thus we obtain
\begin{align*} 
c_{n,u,m}&=\pm{\bigcirc_{s^{\prime}=1}^{r}\,\,\varphi_{1, {t_{m,s^{\prime}}},x,s^{\prime}}}\left({\prod_{s=1}^{r}} \left[t^{u}_{m,s} \cdot (t_{m,s}-1)^{rn}\right] \times{\prod_{\substack{1\leq
s_1<s_2\leq r}}}(t_{m,s_{2}}-t_{m,s_{1}})\right)c_{n,u+r(n+1),m-1}\\ 
&=\pm \prod_{i=1}^{m}\left({\bigcirc_{s^{\prime}=1}^{r}\,\,\varphi_{1,t_{s'},x,s^{\prime}}}\left(\prod_{s=1}^{r} 
 \left[t^{u+(i-1)r(n+1)}_{s} \cdot (t_{s}-1)^{rn}\right] \times{\prod_{\substack{1\leq
s_1<s_2\leq r}}}(t_{s_{2}}-t_{s_{1}})
\right)\right). 
\end{align*}
We are then in a position to conclude.
Indeed, using the definition of the operators $\varphi_{1,t_s,x,s}$, the composition of these operators is nothing but an integral over $[0,1]^r$. More precisely,
\begin{align*}
&{\bigcirc_{s^{\prime}=1}^{r}\,\,\varphi_{1,t_{s'},x,s^{\prime}}}\left({\prod_{s=1}^{r}} \left[t^{u}_{s} \cdot (t_{s}-1)^{rn}\right] \cdot {\prod_{\substack{1\leq
s_1<s_2\leq r}}}(t_{s_{2}}-t_{s_{1}})\right)\\
&={\prod_{s^{\prime}=1}^{r}}\dfrac{1}{{(s^{\prime}-1)!}}\int^1_{0}\cdots\int^1_{0}{\prod_{s=1}^{r}} \left[t^{u+x}_{s} (t_{s}-1)^{rn}{\rm{log\,}}^{s-1}\dfrac{1}{t_{s}}\right] \cdot {\prod_{\substack{1\leq
s_1<s_2\leq r}}}(t_{s_{2}}-t_{s_{1}})\prod_{s=1}^{r}dt_{s}\,,
\end{align*}
then a direct computation enables us to show this last integral does not vanish,
which yields Proposition $\ref{non zero det}$.

The statement of Theorem \ref{Lerch} now follows from Proposition $\ref{critere version II}$, since the determinant is a non-vanishing algebraic constant.

\section{Examples}

We show here three examples of linearly independent polylogarithms, which are  shown by our criterion.
\begin{example}
Put $r=m=10$ and $x=0$. Take $\boldsymbol{\alpha}:=(1,1/2,\ldots,1/10)$ and $\beta=b$ with $|b|\ge e^{2715}$. Then we have $D(\boldsymbol{\alpha},b)=d_{10}=2520$. 
Since we have the inequalities:
\begin{align*}
{\rm{log\,}}2520<7.84, \ {\rm{log\,}}3<1.10, \ {\rm{log\,}}(5/2)<0.92,\,\,\,
\end{align*}
we have
$$\,\,\,{\rm{log\,}}|b|>100(10+{\rm{log\,}}2520+10\,{\rm{log\,}}(5/2))+10\,{\rm{log\,}}3\,.$$
Then the $10^{2}+1$ numbers 
\begin{align*}
1, {\rm{Li}}_1(1/b),\ldots, {\rm{Li}}_{10}(1/b),\ldots, {\rm{Li}}_1(1/(10b)),\ldots, {\rm{Li}}_{10}(1/(10b)),
\end{align*}
are linearly independent over $\Q$. 
\end{example}
\begin{example}
Let $k\ge 2$ be an integer, set $r=m=10^k$, $x=0$. Take $\boldsymbol{\alpha}:=(j)_{1\le j\le 10^k}$ and $\beta=b\in \Z$ with $|b|\ge {\rm{exp}}(2\cdot 10^{3k})$. 
Since $k\ge 2$, we can easily verify 
\begin{align*}
{\rm{log\,}}|b|&>(rm+1){\rm{log\,}}10^k+\left(r^2m(1+{\rm{log\,}}(5/2))+r{\rm{log\,}}3\right)\\
              &=k(10^{2k}+1){\rm{log\,}}10+\left(10^{3k}(1+{\rm{log\,}}(5/2))+10^k{\rm{log\,}}3\right)\enspace.
\end{align*}
Then 
the $10^{2k}+1$ numbers 
\begin{align*} 
1, {\rm{Li}}_1(1/b),\ldots, {\rm{Li}}_{10^k}(1/b),\ldots, {\rm{Li}}_1(10^k/b),\ldots, {\rm{Li}}_{10^k}(10^k/b)\enspace,
\end{align*}
are linearly independent over $\Q$. For instance, we take $r=m=10^4$ and $b=3^{2\cdot 10^{12}}$ then the $10^{8}+1$ numbers
\begin{align*}
1, {\rm{Li}}_1(1/3^{2\cdot 10^{12}}),\ldots, {\rm{Li}}_{10^4}(1/3^{2\cdot 10^{12}}),\ldots, {\rm{Li}}_1(10^4/3^{2\cdot 10^{12}}),\ldots, {\rm{Li}}_{10^4}(10^4/3^{2\cdot 10^{12}}),
\end{align*}
are all linearly independent over $\Q$. 
\end{example}
{\begin{example}
Let $M\ge 5$ be a natural number. Define the polynomial $$f_M(X):=\left(2+\dfrac{1}{M}\right)X^2-2X+\dfrac{2}{M}.$$ Then $X=(M\pm \sqrt{M^2-4M-2})/(2M+1)$ are roots of $f_M(X)$.
Put  $\beta:=(2M+1)/(M-\sqrt{M^2-4M-2})$, $K:=\Q(\beta)$ and $\delta:=e^{7908}$. 
We take $r=m=10$, $\boldsymbol{\alpha}:=(1,1/2,\ldots,1/10)$ and $$M>\dfrac{2\delta^2-\delta+1+\sqrt{4\delta^4+4\delta^3-3\delta^2-6\delta+5}}{4\delta-4}.$$
Then we have 
\begin{align*}V(\boldsymbol{\alpha},\beta,0)=\mathbb{A}(\boldsymbol{\alpha},\beta,0)-\mathcal{A}^{(2)}(\boldsymbol{\alpha},\beta,0)>{\rm{log}}|\beta|-7908>0.
\end{align*}
Thus by Theorem $\ref{Lerch}$, the $10^2+1$ numbers
\begin{align*} 
1, {\rm{Li}}_1(1/\beta),\ldots, {\rm{Li}}_{10^k}(1/\beta),\ldots, {\rm{Li}}_1(1/10\beta),\ldots, {\rm{Li}}_{10^k}(1/10\beta)\enspace,
\end{align*}
are linearly independent over $K$.
For example,we take $M\ge e^{15817}$, 
 the $10^2+1$ numbers
\begin{align*} 
1, {\rm{Li}}_1(1/\beta),\ldots, {\rm{Li}}_{10^k}(1/\beta),\ldots, {\rm{Li}}_1(1/10\beta),\ldots, {\rm{Li}}_{10^k}(1/10\beta)\enspace,
\end{align*}
are linearly independent over $K$.
\end{example}} 

\bigskip

{\bf Acknowledgement}\\
We sincerely thank the referee for comments and precise
references which helped us a lot in our comparison with previous results. {This work is
partly supported by JSPS KAKENHI Grant no. 18K03225 and also by
the Research Institute for Mathematical
Sciences, an International Joint Usage/Research Center located in Kyoto
University.}

\bibliography{}

\begin{thebibliography}{999}%



\bibitem{Baker1975}
A.~Baker, {\it Transcendental Number Theory},  
Cambridge Univ. Press, 1975.



\bibitem{ch2}
G.~V.~Chudnovsky,
{\it Pad\'e approximations to the generalized hypergeometric functions I},
J. Math. Pures et Appl.,  \textbf{58}, (1979),  445--476.

\bibitem{ch3}
G.~V.~Chudnovsky,
{\it Measures of irrationality, transcendence and algebraic independence, Recent progress},
London Math. Soc. Lecture Notes series,  \textbf{56}, Cambridge Univ. Press, (1982),  11--82.

\bibitem{ch9}
G.~V.~Chudnovsky,
{\it On the method of Thue-Siegel},
Annals of Math., \textbf{117}, (1983),  325--382.

\bibitem{ch10}
G.~V.~Chudnovsky,
{\it On applications of Diophantine approximations},
Proc. Natl. Acad. Sci. USA \textbf{81} (1984), 7261--7265. 

\bibitem{Chubrothers}
D.~V.~Chudnovsky and G.~V.~Chudnovsky,
{\it Applications of Pad\'e approximations to diophantine inequalities in values of $G$-functions},
in: Number Theory, New York 1983--84, eds. D.~V.~Chudnovsky, G.~V.~Chudnovsky, H.~Cohn,  M.~B.~Nathanson,
Lecture Notes in Math., \textbf{1135},  Springer, (1985),
9--51.

\bibitem{DHK1}
S.~David, N.~Hirata-Kohno  and M.~Kawashima,
{\it Linear forms in polylogarithms},
preprint.

\bibitem{DHK2}
S.~David, N.~Hirata-Kohno  and M.~Kawashima,
{\it Linear independence criterion of the Lerch functions with different shifts at distinct algebraic points},
preprint.

\bibitem{fe1}
N.~I.~Fel'dman,
{\it Improved estimate for a linear form of logarithms of algebraic numbers},
Mat. Sb., \textbf{77}, (1968), 256--270; English transl. in Math.\ USSR Sbornik, \textbf{6} (1968), 393--406.

\bibitem{FN}
N. I. Fel'dman and Yu. V. Nesterenko (authors),  A. N. Parshin and I. R. Schfarevich (eds.),
Number Theory IV, Encyclopaedia of Mathematical Sciences Vol. 44,
1998.

\bibitem{FSZ}
{S.~Fischler, J.~Sprang and W.~Zudilin,
{\it Many odd zeta values are irrational}, Compositio Math.,
\textbf{155},  (2019) , 938--952. }

\bibitem{Fli}  Y. Z. Flicker, {\it
On $p$-adic $G$-functions},
J. London Math. Soc. (2) \textbf{15},  n$^{\mathrm o}$. 3,  (1977), 395--402.

\bibitem{G1}
A.~I.~Galochkin,
{\it Estimates from below of polynomials in the values of analytic functions of a certain class},
Mat. Sb., \textbf{95 (137)}, no. 3, (1974), 396--417; English transl. in Math.\ USSR Sbornik, \textbf{24}, no. 3,  (1974), 385--407.

\bibitem{G2}
A.~I.~Galochkin,
{\it Lower bounds of  linear forms of values of $G$-functions},
Mat. Zametki, \textbf{18}, no. 4, (1975), 541--552; English transl. in Math. Note, \textbf{18} (1975), 911--917.

\bibitem{Ha}
M.~Hata,
{\it On the linear independence of the values of polylogarithmic functions},
J. Math. Pures et Appl.,  \textbf{69}, (1990),  133--173.

\bibitem{Ha1993}
M.~Hata,
{\it Rational approximations to the dilogarithms},
Trans. Amer. Math. Soc., \textbf{336}, \ no.~1,  (1993),  363--387.

\bibitem{H-I-W}
N.~Hirata-Kohno, M.~Ito and Y.~Washio,
{\it A criterion for the linear independence of polylogarithms over a number field},
RIMS Kokyouroku Bessatu, \textbf{64}, (2017), 3--18.

\bibitem{H-K-S}
M.~Hirose, M.~Kawashima and N.~Sato,
{\it A lower bound of the dimension of the vector space spanned by the special values of certain functions},
Tokyo J. ~Math., \textbf{40},\ no.~2, (2017), 439--479.

\bibitem{Ka}
M.~Kawashima,
{\it Evaluation of the dimension of the $\Q$-vector space spanned by the special values of the Lerch function},
Tsukuba J. ~Math. \textbf{38},\ no.~2, {(2014)}, 171--188.


\bibitem{marc}
{R.~Marcovecchio,
{\it Linear independence of forms in polylogarithms},
Ann.  Scuola Nor.  Sup. Pisa CL. Sci., \textbf{5}, (2006), 1--11.}

\bibitem{Mi} 
{M.~A.~Miladi,
{\it R\'ecurrences lin\'eaires et approximations simultan\'ees de type Pad\'e: applications \`a l'arithm\'etiqus},
Th\`ese, Universit\'e des S. et T. de Lille, 2001.}


\bibitem{N}
E.~M.~Niki\v{s}in,
{\it On irrationality of the values of the functions $F(x,~s)$},
Math.\ USSR Sbornik, \textbf{37}, \ no.~3,  (1980),  381--388 (originally published in Mat, Sb., \textbf{109}, \ no.~3,  (1979)).


\bibitem{N0}
E.~M.~Niki\v{s}in,
{\it On logarithms of natural numbers},
Math.\ USSR Izvestia, \textbf{15}, \ no.~3,  (1980),  523--530 (originally published in Izv. Akad. Nauk., \textbf{43}, \ no.~6,  (1979)).


\bibitem{Nurma} M. S. Nurmagomedov,  {\it The arithmetical properties of the values of $G$-functions},  
Vestnik Moskov Univ. Ser. I Mat. Meh., \textbf{26},  n$^{\mathrm o}$. 6, (1971), 79--86.

\bibitem{R-T}
G.~Rhin and  P.~Toffin,
{\it Approximants de Pad\'{e} simultan\'{e}s de logarithmes},
J. Number Theory, \textbf{24},  (1986), 284--297.

\bibitem{RV0}
G.~Rhin and  C.~Viola,
{\it On a permutation group related to $\zeta (2)$}, Acta Arith., \textbf{77}, \ no.~1,  
 (1996),  23--56.

\bibitem{RV1}
G.~Rhin and  C.~Viola,
{\it The permutation group method for the dilogarithms},
Ann.  Scuola Nor.  Sup. Pisa CL. Sci., \textbf{4},  \ no.~3, (2005), 389--437.


\bibitem{Ri}
{T. Rivoal,
{\it Ind\'ependance lin\'eaire des valeurs des polylogarithmes},
J. Th\'eorie des Nombres Bordeaux,  \textbf{15}, \ no.~2,  (2003), 551--559.}



\bibitem{Siegel}
C.~Siegel,
{\it Uber einige Anwendungen diophantischer Approximationen}, 
Abhandlungen der Preu$\ss$ischen Akademie der Wissenschaften.
Physikalisch Mathematische Kalasse 1929, Nr.~1.


\bibitem{Va}
K.~V$\ddot{\text{a}}$$\ddot{\text{a}}$n$\ddot{\text{a}}$nen,
{\it On linear forms of a certain class of $G$-functions},
Acta Arith., Vol. \textbf{36}, (1980), 273--295.

\bibitem{Va-Gu}
K.~V$\ddot{\text{a}}$$\ddot{\text{a}}$n$\ddot{\text{a}}$nen, G. Xu
{\it On linear forms of $G$-functions},
Acta Arith., Vol. \textbf{50}, (1988), 251--263.

\bibitem{VZ1}
C.~Viola and W. Zudilin,
{\it Linear independence of dilogarithmic values},
J. Reine Angew Math., \textbf{736}, (2018),  193--223.

\bibitem{Z}
{W. Zudilin,
{\it On a measure of irrationality for values of $G$-functions},
Math. USSR Izv., \textbf{60}, no.\ 1, (1996), 91--118. }
\end{thebibliography}

\begin{scriptsize}
\begin{minipage}[t]{0.38\textwidth}

Sinnou David,
\\sinnou.david@imj-prg.fr
\\Institut de Math\'ematiques
\\de Jussieu-Paris Rive Gauche
\\CNRS UMR 7586,\\
Sorbonne Universit\'{e}, Paris, France\\
\& CNRS UMI 2000 Relax\\
 Chennai Mathematical Institute\\
 Kelambakkam, India \\\\
\end{minipage}
\begin{minipage}[t]{0.32\textwidth}
Noriko Hirata-Kohno, 
\\hirata@math.cst.nihon-u.ac.jp
\\Department of Mathematics
\\College of Science \& Technology
\\Nihon University
\\Tokyo, Japan\\\\
\end{minipage}
\begin{minipage}[t]{0.3\textwidth}
Makoto Kawashima,
\\kawashima.makoto@nihon-u.ac.jp
\\Department of Liberal Arts \\and Basic Sciences
\\College of Industrial Engineering
\\Nihon University, Chiba, Japan\\\\
\end{minipage}

\end{scriptsize}

\end{document}